\documentclass[12pt]{amsart}

\usepackage{amsmath}
\usepackage[margin=1.00in]{geometry}
\usepackage{setspace}
\usepackage{nicefrac}
\usepackage{microtype}

\usepackage{amsthm}
\theoremstyle{plain}
\newtheorem{theorem}{Theorem}
\newtheorem{lemma}[theorem]{Lemma}

\theoremstyle{definition}
\newtheorem{example}{Example}
\newtheorem{definition}{Definition}
\newtheorem{theorem-definition}{Theorem-Definition}
\newtheorem*{theorem-definition*}{Theorem-Definition}
\newtheorem{remark}{Remark}
\newtheorem*{heuristic*}{Heuristic}

\numberwithin{equation}{section}

\newcommand{\oD}{D}
\newcommand{\pD}{\Delta}

\newcommand{\floor}[1]{\left\lfloor #1 \right\rfloor}
\newcommand{\nubdpp}{\nu_{\textrm{BDPP}}}

\DeclareMathOperator{\vol}{vol}
\DeclareMathOperator{\Bigc}{Big}
\DeclareMathOperator{\Effb}{\overline{Eff}}

\newcommand{\rat}{\dashrightarrow}
\newcommand{\set}[1]{\left\{ #1 \right\}}
\newcommand{\abs}[1]{\left\lvert #1 \right\rvert}
\newcommand{\ang}[1]{\left\langle #1 \right\rangle}

\newcommand{\R}{\mathbb R}
\renewcommand{\P}{\mathbb P}
\newcommand{\cO}{\mathcal O}

\title[Numerical dimensions do not coincide]{Notions of numerical Iitaka dimension do not coincide}
\author{John Lesieutre}
\address{The Pennsylvania State University \\ 204 McAllister Building \\ University Park, PA 16801}
\email{jdl@psu.edu}

\begin{document}

\begin{abstract}
Let \(X\) be a smooth projective variety.  The Iitaka dimension of a divisor \(D\) is an important invariant, but it does not only depend on the numerical class of \(D\).  However, there are several definitions of ``numerical Iitaka dimension'', depending only on the numerical class.  In this note, we show that there exists a pseuodoeffective \(\R\)-divisor for which these invariants take different values.  The key is the construction of an example of a pseudoeffective \(\R\)-divisor \(\oD_+\) for which \(h^0(X,\floor{m \oD_+}+A)\) is bounded above and below by multiples of \(m^{3/2}\) for any sufficiently ample \(A\).
\end{abstract}

\maketitle

\section{Introduction}

Given a divisor \(D\) on a projective variety \(X\), the Iitaka dimension of \(D\) is a fundamental invariant measuring the asymptotic growth of spaces of sections of \(mD\).

\begin{theorem-definition*}[{e.g.\ \cite[Corollary 2.1.38]{lazarsfeld}}]
Suppose that \(X\) is a smooth projective variety and \(D\) is a divisor on \(X\).  There exists an integer \(\kappa(D)\), the \emph{Iitaka dimension} of \(D\), as well as constants \(C_1,C_2 > 0\) such that for sufficiently large and divisible \(m\),
\[
C_1 m^{\kappa(D)} < h^0(X,mD) < C_2 m^{\kappa(D)}.
\]
\end{theorem-definition*}

The most important case is when \(D = K_X\) is the canonical class, in which case \(\kappa(K_X)\) is simply the Kodaira dimension of \(X\).

The Iitaka dimension has the inconvenient property that it is not a numerical invariant of \(D\).  It is possible, for example, that there exist two divisors \(\oD_1\) and \(\oD_2\) which have the same numerical class, but such that any multiple of \(\oD_1\) is rigid, while \(\oD_2\) moves in a pencil.  In this case, \(\kappa(\oD_1) = 0\) while \(\kappa(\oD_2) \geq 1\)~\cite[Example 6.1]{lehmann}.

One approach to constructing a numerical analog of the Iitaka dimension is to perturb each \(mD\) by a fixed ample divisor \(A\), considering the dimensions \(h^0(X,mD+A)\) as \(m\) increases.  This growth of these sections does indeed yield an important numerical invariant, Nakayama's \(\kappa_\sigma(D)\).  There are a number of other possible definitions of numerical dimension, some of which we recall in the next section.

The main result of this paper is that, at least when \(D\) is an \(\R\)-divisor, the spaces of sections \(h^0(X,\floor{mD}+A)\) need not even grow polynomially in \(m\).  

\begin{theorem}
There exists a smooth projective threefold \(X\) and a pseudoeffective \(\R\)-divisor \(D\) on \(X\) such that for any sufficiently ample class \(A\), there exist constants \(C_1,C_2 > 0\) so that
\[
C_1 m^{3/2} < h^0(X,\floor{mD}+A) < C_2 m^{3/2}.
\]
\end{theorem}

As a consequence of this calculation, we conclude that various notions of numerical dimension do not coincide in general, contrary to general expectation.  The example is a pseudoeffective \(\R\)-divisor on a Calabi--Yau threefold \(X\) which has previously appeared in the work of Oguiso~\cite{oguiso}.

\section{Preliminaries}
\label{preliminaries}
 
We begin with some preliminary definitions. 
We work throughout over an algebraically closed field \(K\) of characteristic \(0\). Write \(\equiv\) for the relation of numerical equivalence and \(N^1(X)\) for the finite-dimensional \(\R\)-vector space of numerical classes of divisors on \(X\).  If \(D\) is a Cartier divisor, we will write \(h^0(X,D)\) for \(h^0(X,\cO_X(D))\).

\begin{definition}[{\cite[Ch.\ 5]{nakayama}}]
\label{kappasigma}
The numerical dimension \(\kappa_\sigma(D)\) is the largest integer \(k\) such that for some ample divisor \(A\), one has
\[
\limsup_{m \to \infty} \frac{h^0(X,\floor{mD}+A)}{m^k} > 0.
\]

If no such \(k\) exists, we take \(\kappa_\sigma(D) = -\infty\).
We will also consider a closely related inariant: \(\kappa_\sigma^\R(D)\) is the largest real number for which this inequality holds, so that  \(\kappa_\sigma(D) = \floor{ \kappa_\sigma^\R(D)}\).    It will follow from our example that these two quantities may be distinct.
\end{definition}

\begin{remark}
There are several variations on this definition.  For example, one might replace the \(\limsup\) by a \(\liminf\); this is the definition of \(\kappa_\sigma\) used in~\cite{cascinietal} and some older versions of~\cite{nakayama}. Nakayama denotes this invariant by \(\kappa_\sigma^-\).  It remains unclear whether these values can be distinct.

It is also possible to ask for the smallest integer \(k\) for which
\[
\limsup_{m \to \infty} \frac{h^0(X,\floor{mD}+A)}{m^k} < \infty.
\]
Nakayama denotes the resulting invariant by \(\kappa_\sigma^+(D)\).  This is the version of numerical dimension used in, for example, \cite{kawamatanumdimzero}.  Our main example shows that this invariant is not equal to \(\kappa_\sigma(D)\) in general.
\end{remark}

An important result of Nakayama~\cite[Theorem V.1.12]{nakayama} states that if \(D\) is a pseudoeffective \(\R\)-divisor on \(X\) for which \(h^0(X,\floor{mD}+A)\) is not bounded in \(m\) (i.e.\ for which \(D \neq N_\sigma(D)\)), then for any sufficiently ample divisor \(A\) there is a constant \(C\) for which
\[
h^0(X,\cO_X(\floor{mD}+A)) > Cm
\]
for all \(m\). The same result has been recovered in positive characteristic~\cite{cascinietal}.   It follows that if \(h^0(X,\mathcal O_X(\floor{mD} +A ))\) is not bounded, then \(\kappa_\sigma^\R(D) \geq 1\).

A second definition of numerical dimension, Nakayama's \(\kappa_\nu(D)\), is based on the notion of \emph{numerical domination}.

\begin{definition}[{\cite[Ch.\ 5, \S 2]{nakayama}}, cf.~\cite{eckl}]
\label{numericaldomination}
Suppose that \(D\) is a pseudoeffective \(\R\)-divisor on \(X\) and \(W \subset X\) is a subvariety.  We say that \(D\) \emph{numerically dominates} \(W\) (written \(D \succeq W\)) if there exists a birational morphism \(\pi : \tilde{X} \to X\) such that \(\pi^{-1}\mathcal I_W \cdot \cO_{\tilde{X}} = \cO_{\tilde{X}}(E)\) and for every positive \(b\), there exist \(x > b\) and \(y > b\) such that the class \(x \cdot \pi^\ast D - y \cdot E_W + A\) is pseudoeffective.  
\end{definition}

For discussion of this condition and some illuminating illustrations, we refer to the works of Nakayama~\cite{nakayama} and Eckl~\cite{eckl}.

\begin{definition}[{\cite{nakayama}}]
The numerical dimension \(\kappa_\nu(D)\) is the minimum dimension of a subvariety \(W \subset X\) for which \(D\) does not numerically dominate \(W\).
\end{definition}

A third definition is provided in terms of the positive intersection product.  While we refer to \cite{bdpp} and \cite{bfj} for the details of the construction, to a set of pseudoeffective divisors \(\oD_1,\ldots,\oD_k\) on \(X\) one  associates a class in \(N^k(X)\) which roughly measures the class of the intersection among the \(\oD_i\) which takes place away from their base loci.  This positive intersection product is continuous on the big cone, but unlike the usual intersection form, is not linear.

\begin{definition}[\cite{bdpp}]
\label{movingdef}
The numerical dimension \(\nubdpp(D)\) is the largest integer \(k\) for which the positive intersection product \(\ang{D^k}\) is nonzero.
\end{definition}

\begin{remark}
In the case that \(D\) is nef, the positive intersection product coincides with the usual intersection form,  and Definition~\ref{movingdef} coincides with the original definition of Kawamata~\cite{kawamataabundance}.  In this case, it is proved by Nakayama that \(\kappa_\sigma(D) = \kappa_\nu(D) = \nubdpp(D)\).
\end{remark}

\section{Main example}

\begin{example}[{\cite[\S 6]{oguiso}}]
Let \(X\) be a smooth threefold in \(\P^3 \times \P^3\) given as the intersection of general divisors of bidegrees \((1,1)\), \((1,1)\), and \((2,2)\).  It follows from adjunction and the Lefschetz hyperplane theorem that \(X\) is a smooth, Calabi-Yau threefold of Picard rank \(2\).  Let \(\pi_i : X \to \P^3\) (\(i = 1,2\)) be the two projections.  A basis for \(N^1(X)\) is given by the two classes \(H_i = \pi_i^\ast \mathcal O_{\P^3}(1)\).

The maps \(\pi_1\) and \(\pi_2\) are both generically \(2\) to \(1\), and so there are two associated birational covering involutions \(\tau_i : X \rat X\).  The maps \(\tau_i\) are not biregular, since the \(\pi_i\) have some positive-dimensional fibers.  However, since \(K_X\) is trivial, these maps extend to pseudoautomorphisms of \(X\), i.e.\ birational maps which are an isomorphism in codimension \(1\).

Oguiso checks that with respect to the basis \(H_1,H_2\), we have:
\[
\tau_1^\ast = \begin{pmatrix} 1 & 6 \\ 0 & -1 \end{pmatrix}, \quad
\tau_2^\ast = \begin{pmatrix} -1 & 0 \\ 6 & 1 \end{pmatrix}.
\]
The composite map \(\phi = \tau_1 \circ \tau_2\) acts on \(N^1(X)\) by 
\[
\phi^\ast = \tau_2^\ast \tau_1^\ast = \begin{pmatrix}
-1 & -6 \\ 6 & 35
\end{pmatrix}.
\]
Recall that for \(0 \leq k \leq N\), the \(k^\text{th}\) dynamical degree of \(\phi\) is the number
\[
\lambda_k(\phi) = \lim_{n \to \infty} \left( (\phi^n)^\ast(H^k) \cdot H^{N-k} \right)^{1/n},
\]
where \(H\) is a fixed ample divisor; in fact this limit exists and is independent of \(H\)~\cite{dinhsibony}.  In our case, the first dynamical degree is the spectral radius of \(\phi^\ast\), which is given by
\[
\lambda = \lambda_1(\phi) = 17 + 12 \sqrt{2} \approx 33.97\dots.
\]

It is also useful to compute the nef and pseudoeffective cones, as well as certain subcones.  The nef cone is spanned by the classes of the two divisors \(H_1\) and \(H_2\), while the pseudoeffective cone coincides with the movable cone and is spanned by the two eigenvectors of \(\phi^\ast\), which up to a choice of normalization are given by:  
\begin{align*}
  \pD_+ &= (1 - \sqrt{2}) H_1 + (1 + \sqrt{2}) H_2, \\
  \pD_- &= (1 + \sqrt{2}) H_1 + (1 - \sqrt{2}) H_2.
\end{align*}

These satisfy \(\phi^\ast \pD_+ = \lambda \pD_+\) and \(\phi^\ast \pD_- = \lambda^{-1} \pD_-\).
Let \(\oD_+\) and \(\oD_-\) be the two \(\R\)-divisors in the span of \(H_1\) and \(H_2\) which represent these classes.  It is necessary to choose explicit \(\R\)-divisors rather than numerical classes in order to make sense of the round-downs \(\floor{m\oD_+}\), but the result \(\kappa_\sigma(\pD_+)\) is ultimately independent of the choice.

It will also be convenient for us to work with the cone \(\mathcal C \subset N^1(X)\) spanned by \(H_2\) and \(\tau_1^\ast H_2 = 6H_1 - H_2\).  This cone has the property that if \(D\) is any divisor class lying in \(\mathcal C\), then either \(D\) or \(\tau_1^\ast D\) is big and nef.
\end{example}

\begin{theorem}
\label{oguisocomputation}
The pseudoeffective \(\R\)-divisor \(\oD_+\) satisfies:
\begin{enumerate}
\item \(\nubdpp(\oD_+) = \kappa_\sigma(\oD_+) =  1\);
\item \(\kappa_\sigma^\R(\oD_+) = \nicefrac{3}{2}\);
\item \(\kappa_\sigma^+(\oD_+) = \kappa_\nu(\oD_+) = 2\).
\end{enumerate}
\end{theorem}

The bulk of the work is dedicated to computing \(h^0(X,\floor{m\oD_+}+A)\) and hence \(\kappa_\sigma^\R(\oD_+)\); in fact, the computations of \(\nubdpp(\oD_+)\) and \(\kappa_\nu(\oD_+)\) follow from this and the inequalities of~\cite{lehmann} and \cite{eckl}.  Since these can also be computed directly, we include a derivation for the sake of completeness.  The main complication is that the definition of \(\kappa_\sigma\) and \(\kappa_\sigma^\R\) for \(\R\)-divisors requires working  with round-downs, while the other notions do not; this makes it somewhat tedious to compute.

\begin{heuristic*}
Before giving a proof, we briefly explain the calculation of \(h^0(X,\floor{m\oD_+}+A)\).  The variety \(X\) has the property that given any big divisor class \(D\), there is a pseudoautomorphism (either \(\phi^m\) or \(\tau_1 \circ \phi^m\)), such that the pullback of \(D\) under this map is big and nef.  Since \(h^0(X,D)\) is invariant under pulling back by a pseudoautomorphism, and \(h^0(X,A)\) can be computed using the Riemann--Roch theorem if \(A\) is big and nef, it is possible to compute \(h^0(X,D)\) for any big divisor \(D\), even those such as \(\floor{m\oD_+}+A\) which have complicated base loci and lie very close to the pseudoeffective boundary.

For simplicity, we work in the basis for \(N^1(X)\) given by \(\pD_+\) and \(\pD_-\), the two extremal rays on \(\Effb(X)\).  The pullback \(\phi^\ast\) is given in this basis by \(\left( \begin{smallmatrix} \lambda & 0 \\ 0 & \lambda^{-1} \end{smallmatrix} \right)\), and so it preserves a quadratic form, the product of the two coordinates of a class written with respect to this basis.
Choosing a suitable scaling of \(\pD_+\), we may assume that \(A = \pD_+ + \pD_-\) is ample.  With respect to this basis, the class \(m\pD_+ + A\) has coordinates \((m+1,1)\). The ample cone consists of divisors for which the two coordinates are approximately equal (more precisely, for which their ratio is contained in some bounded interval).  Since pullback by \(\phi^\ast\) preserves the product of the coordinates, the pullback \(\phi^{k_m}(m\pD_++A)\) which is ample must be roughly \(\sqrt{m} \pD_+ + \sqrt{m} \pD_-\), which is the case when \(k_m \approx -\frac12 \log_\lambda m\). We are then in position to compute
\begin{align*}
  h^0(X,\floor{m\oD_+}+A) &=  h^0(X,\phi^{\ast k_m}(\floor{m\oD_+}+A)) = \chi(\phi^{\ast k_m}(\floor{m\oD_+}+A)) \\
  &\approx (\phi^{\ast k_m}(\floor{m\oD_+}+A))^3/6 \approx  (\phi^{\ast k_m}(m\pD_++A))^3/6 \\
  &\approx \left( \sqrt{m} \pD_+ + \sqrt{m} \pD_- \right)^3/6 = C m^{3/2}.
\end{align*}
\end{heuristic*}

The next few lemmas establish the required bounds required to make this precise. For simplicity, we focus our computations on the particular variety \(X\), but similar results can be obtained for more general contexts; see Lemma~\ref{vvolcompute}.  The proofs involve many constants whose precise values are not important; we will denote these constants by \(C_i\), \(C_{1,j}\) and \(C_{2,k}\) as they appear.

It is convenient to introduce a new set of coordinates on \(\Bigc(X)\). Given a big class \(D = a_1 \pD_+ + a_2 \pD_-\) (which must have \(a_1,a_2 > 0\)), we set
\[
L_1(D) = a_1 a_2, \quad L_2(D) = \frac{a_1}{a_2}.
\]
For an \(\R\)-divisor \(\oD\), we write \(L_i(\oD)\) for the corresponding value for the numerical class.
These coordinates owe their convenience to the facts that
\[
L_1(\phi^\ast D) = L_1(D), \quad L_2(\phi^\ast D) = \lambda^2 L_2(D).
\]

\begin{lemma}
\label{landinc}
Suppose that \(D\) is a big class on \(X\).  Then there exists an integer \(k\) so that \((\phi^\ast)^k(D)\) lies in the cone \(\mathcal C\).
\end{lemma}

\begin{proof}
The cone \(\mathcal C\) is bounded by the two divisors
\begin{align*}
  H_2 &= \left( \frac{2+\sqrt{2}}{8} \right) \pD_+ + \left( \frac{2-\sqrt{2}}{8} \right) \pD_-, \\
  \tau_1^\ast H_2 &= \left( \frac{10 - 7\sqrt{2}}{8} \right) \pD_+ + \left( \frac{10 + 7\sqrt{2}}{8} \right) \pD_-,
\end{align*}
and so
\begin{align*}
\label{oguiso:CboundL2_1}
L_2(H_2) &= \frac{2+\sqrt{2}}{2 - \sqrt{2}} = 3 + 2\sqrt{2} \\
L_2(\tau_1^\ast H_2) &= \frac{10-7\sqrt{2}}{10+7\sqrt{2}} = 99 - 70\sqrt{2}.
\end{align*}
Then
\[
\label{ogusio:Cbound_ratio}
\frac{L_2(H_2)}{L_2(\tau_1^\ast H_2)} = \frac{3 + 2\sqrt{2}}{99-70\sqrt{2}} = 577 + 408 \sqrt{2} = \lambda^2.
\]

We have seen that \(L_2(\phi^\ast D) = \lambda^2 L_2(D)\), and the claim follows: explicitly, we may take
\[
k = - \floor{ \frac{1}{2} \left( \log_\lambda L_2(D) - \log_\lambda L_2(\tau_1^\ast H_2) \right) }. \qedhere
\]
\end{proof}

The next observation is that on this variety \(X\), it is straightforward to compute \(h^0(X,D)\) for any big and nef \(D\).

\begin{lemma}
\label{riemannroch}
There exist constants \(C_{1,1}, C_{2,1} > 0\) such that if \(\oD = a_1 H_1 + a_2 H_2\) is any big and nef Cartier divisor,
\[
C_{1,1} \oD^3 < h^0(X,\oD) < C_{2,1} \oD^3.
\]
\end{lemma}

\begin{proof}
The intersection form on divisors on \(X\) is given by \(H_1^3 = H_2^3 = 2\) and \(H_1^2 H_2 = H_1 H_2^2 = 6\). Since  \(X\) is a Calabi--Yau threefold, it follows from the Hirzebruch--Riemann--Roch theorem and Kawamata--Viehweg vanishing that for any big and nef class \(D\),
\begin{align*}
h^0(X,D) &= \chi(X,D) = \frac{D^3}{6} + \frac{c_2(X) \cdot D}{12} \\
\frac{h^0(X,D)}{D^3} &= \frac16 + \frac{1}{12} \frac{c_2(X) \cdot D}{D^3}.
\end{align*}
We have \(c_2(X) = H_1^2 + 6 H_1 H_2 + H_2^2\), and so explicitly,
\begin{align*}
\frac{h^0(X,D)}{D^3} &= \frac{1}{6} + \frac{1}{12} \frac{(H_1^2 + 6 H_1 H_2 + H_2^2) \cdot (a_1 H_1 + a_2 H_2)}{(a_1 H_1 + a_2 H_2)^3} \\
  &=\frac{1}{6} + \frac{1}{12} \frac{ 44(a_1 + a_2) }{2a_1^3 + 18 a_1^2 a_2 + 18 a_1 a_2^2 + 2a_2^3} 
 = \frac{1}{6} + \frac{11}{6} \frac{1}{(a_1 + a_2)^2 + 6a_1 a_2}.
\end{align*}
Since \(a_1\) and \(a_2\) are non-negative integers, not both \(0\), the claim holds with \(C_{1,1} = 1/6\) and \(C_{2,1} = 2\).
\end{proof}

\begin{lemma}
\label{computeh0}
There exist constants \(C_{1,2}, C_{2,2} > 0\) such that if \(\oD\) is any Cartier divisor contained in the cone \(\mathcal C\), 
\[
C_{1,2} L_1(\oD)^{3/2} < h^0(X,\oD) < C_{2,2} L_1(\oD)^{3/2}.
\]
\end{lemma}

\begin{proof}
We may write \( D = a_1 \pD_+ + a_2 \pD_-\) where \(a_1 = (L_1(D) L_2(D))^{1/2}\) and \(a_2 = (L_1(D) / L_2(D))^{1/2}\).  Set \(\beta_1 = \pD_+^2 \cdot \pD_-\) and \(\beta_2 = \pD_+ \cdot \pD_-^2\), which are both greater than \(0\).  We have
\[
D^3 = \beta_1 L_1(D)^{3/2} L_2(D)^{1/2} + \beta_2 L_1(D)^{3/2} L_2(D)^{-1/2}.
\]
There are nonzero constants \(\alpha_1\) and \(\alpha_2\) so that the cone \(\mathcal C\) is defined by  \(\alpha_1 \leq L_2(-) \leq \alpha_2\), and so \(D^3\) is bounded above and below by positive multiples of \(L_1(D)^{3/2}\).
If \(D\) is big and nef, the claim follows immediately from Lemma~\ref{riemannroch}.  Otherwise, \(\tau_1^\ast D\) is big and nef.  Since \(L_1(\tau_1^\ast D)\) is bounded above and below by constant multiples of \(L_1(D)\), the claim follows.
\end{proof}

The next lemma checks that rounding down does not have a large impact on \(L_1(\floor{\oD}+A)\).

\begin{lemma}
  \label{L1rounding}
There exist constants \(C_1, C_{1,3}, C_{2,3} > 1\) such that for any big \(\R\)-divisor \(\oD = a_1 \oD_+ + a_2 \oD_-\) and any ample divisor \(A = b_1 \oD_+ + b_2 \oD_-\) with \(b_i > C_1\), we have
\[
C_{1,3} L_1(\oD+A) \leq L_1(\floor{\oD}+A) \leq C_{2,3} L_1(\oD+A)
\]
\end{lemma}

\begin{proof}
  Suppose that \(D = a_1 \pD_+ + a_2 \pD_-\), and that \(\floor{D} = \tilde{a}_1 \pD_+ + \tilde{a}_2 \pD_-\).  It is clear that there is a constant \(C_1\) so that \(\abs{a_i - \tilde{a}_i } < C_1\): to compute the \(\tilde{a}_i\), one expresses the divisor in terms of the basis \(H_1\) and \(H_2\), rounds down the coefficients, and then changes basis back.  Increasing \(C_1\) if necessary, we may assume that \(C_1 > 1\).

Then
\(D+A = (a_1+b_1) \pD_+ + (a_2+b_2) \pD_-\) and \(\floor{D}+A = (\tilde{a}_1+b_1) \pD_+ + (\tilde{a}_2+b_2) \pD_-\), and we find that
\[
\frac{L_1(\floor{D}+A)}{L_1(D+A)} = \frac{(\tilde{a}_1+b_1)(\tilde{a}_2+b_2)}{(a_1+b_1)(a_2+b_2)} =
\frac{\tilde{a}_1+b_1}{a_1+b_1} \cdot \frac{\tilde{a}_2+b_2}{a_2+b_2}
\]
Since \(b_i > C_1 > 1\), both  \(a_i + b_i\) and \(\tilde{a}_i + b_i\) are greater than \(1\), and so
\[
\abs{ \log \left( \frac{ \tilde{a}_i + b_i }{ a_i + b_i } \right) } = \abs{ \log \left( \tilde{a}_i + b_i \right) - 
\log \left( a_i + b_i \right) } < \abs{\tilde{a}_i - a_i } < C_1,
\]
which implies that each of the factors on the right hand side of the preceding equation are bounded by multiplicative factors of \(e^{-C_1}\) and \(e^{C_1}\).  The result follows with \(C_{1,3} = e^{-2C_1}\) and \(C_{2,3} = e^{2C_1}\).
\end{proof}

\begin{theorem}[$\Rightarrow$ Theorem~\ref{oguisocomputation}, (2)]
Suppose that \(A = b_1 \pD_+ + b_2 \pD_-\) is an ample Cartier divisor with \(b_1 , b_2 \geq C_1\).
There exist constants \(C_{1,4}\) and \(C_{2,4}\) such that for all sufficiently large \(m\),
\[
C_{1,4} m^{3/2} < h^0(X,\floor{m \oD} +A) < C_{2,4} m^{3/2}.
\]
\end{theorem}

\begin{proof}
We have
\[
L_1(m\oD_+ + A ) = L_1( (m+b_1) \oD_+ + b_2 \oD_-) = (m+b_1)(b_2).
\]
It follows from Lemma~\ref{L1rounding} that 
\[
C_{1,3} (m+b_1)(b_2) \leq L_1(\floor{m\oD_+}+A)  \leq C_{2,3} (m+b_1)(b_2).
\]
According to Lemma~\ref{landinc}, for every value of \(m\), there exists a constant \(k_m\) for which \((\phi^{k_m})^\ast(\floor{m\oD_+}+A)\) lies in the cone \(\mathcal C\), and since \(L_1(-)\) is invariant under \(\phi\), this shows
\[
C_{1,3} (m+b_1)(b_2) \leq L_1(\phi^{k_m \ast} (\floor{m\oD_+}+A) ) \leq C_{2,3} (m+b_1)(b_2). \]
Since \(h^0(X,\floor{m\pD_+}+A) = h^0(X,\phi^{k_m\ast}(\floor{m\oD_+}+A))\), Lemma~\ref{computeh0} yields
\[
C_{1,2} \left( C_{1,3} (m+b_1)(b_2) \right)^{3/2} \leq h^0(X,\floor{m\oD_+}+A) \leq C_{2,2} \left( C_{2,3}  (m+b_1)(b_2) \right)^{3/2},
\]
and the theorem follows.
\end{proof}

\begin{remark}
For any given value of \(m\), it is straightforward to use a computer algebra system and the Riemann--Roch theorem for a Calabi-Yau threefold to determine the exact value of \(h^0(X,\floor{m\oD_+}+A)\).  This is demonstrated in the accompanying SageMath script \texttt{oguisoexample.sage}, in which we verify that for the ample divisor \(A = H_1 + H_2\), taking \(m = 2^k\) for \(10 \leq k \leq 50\), we have
\[
24 \cdot m^{3/2} < h^0(X,\floor{m\pD_+}+A) < 54 \cdot m^{3/2}.
\]
\end{remark}

The computations \(\kappa_\sigma(\oD_+) = 1\), \(\kappa_\sigma^-(\oD_+) = 1\), \(\kappa_\sigma^\R(\oD_+) = \nicefrac{3}{2}\), and \(\kappa_\sigma^+(\oD_+) = 2\) are immediate.  It remains to compute \(\nubdpp(\oD_+)\) and \(\kappa_\nu(\oD_+)\).

\begin{proof}[Proof of Theorem~\ref{oguisocomputation}, (1)]
If \(\phi : X \rat Y\) is an isomorphism in codimension \(1\), then \(\ang{\phi^\ast \oD_1 \cdot \phi^\ast \oD_2} = \phi^\ast_1(\ang{\oD_1 \cdot \oD_2})\), where \(\phi_1^\ast : N_1(X) \to N_1(X)\) is the pullback map on curve classes.  Then for any value of \(n\), we have
\begin{align*} 
  \ang{(\phi^{n\ast}(\oD_+) + \phi^{n\ast}(\oD_-))^2} &= \phi_1^{n\ast}( \ang{\oD_++\oD_-}^2 )  \\
  \ang{(\lambda^n \oD_+ + \lambda^{-n} \oD_-)^2} &= \phi_1^{n\ast}( \ang{\oD_++\oD_-}^2 ) \\
  \ang{(\oD_+ + \lambda^{-2n} \oD_-)^2} &= \lambda^{-2n} \phi_1^{n\ast}( \ang{\oD_++\oD_-}^2).
\end{align*}
Since \(\phi^\ast_1\) has spectral radius \(\lambda < \lambda^2\), the quantity on the right approaches \(0\).  On the other hand, the classes of the divisors \(\oD_+ + \lambda^{-2n} \oD_-\) approach \(\oD_+\) from an ample direction in \(N^1(X)\). It follows from the definition of the positive intersection product for pseudoeffective classes~\cite[Definition 2.10]{bfj} that the limit of the left side is \(\ang{\oD_+^2}\).
Consequently \(\ang{\oD_+^2} = 0\), and so \(\nubdpp(\oD_+) = 1\).
\end{proof}

\begin{proof}[Proof of Theorem~\ref{oguisocomputation}, (3)]

 Suppose that \(V \subset X\) is a subvariety such that \(\oD_+\) does not numerically dominate \(V\), and let \(\pi : Y \to X\) be the blow-up along \(V\), with exceptional divisor \(E\).  It follows from Definition~\ref{numericaldomination} that there exists a value \(b > 0\), such that for any \(x > b\), the class \(x \pi^\ast \oD_+ - b E + A\) is not pseudoeffective.  If \(H\) is an ample divisor on \(X\), then \(A = \pi^\ast H - \epsilon E \) is ample for sufficiently small \(\epsilon\), and we have
\[
h^0(Y,\pi^\ast( \floor{\oD_+} + H ) - bE) = h^0(Y,\pi^\ast( \floor{\oD_+} ) + A + \epsilon E - bE),
\]
which vanishes if \(\epsilon\) is sufficiently small \(\epsilon\) since \(\pi^\ast( \floor{\oD_+} ) + A - bE\) is not pseudoeffective.

Let \(\mathcal I_V\) be the ideal sheaf of \(V\), and let \(W_b\) be the subscheme cut out by \(\mathcal I_V^b\).  It follows that \(h^0(X,\mathcal I_Z^b( \floor{ m\oD_+ } +A )) = 0\) and so 
\[
H^0(Y,\pi^\ast(\floor{m\oD_+}+A)) \to H^0(W_b,(\floor{m\oD_+}+A)\vert_{W_b})
\]
is injective.  Since the left side grows as a power of \(m^{3/2}\), the dimension of \(W\) must be at least \(2\).  By definition, this means that \(\kappa_\nu(\oD_+) = 2\).
\end{proof}

\begin{remark}
The question of whether \(\kappa_\sigma(D) = \kappa_\nu(D)\) in general originates with Nakayama.  The general equality \(\nubdpp(D) = \kappa_\sigma(D) = \kappa_\nu(D)\) is asserted in the two papers~\cite{lehmann} and~\cite{eckl}.  These papers prove a number of remarkable inequalities between various notions of numerical dimension, but unfortunately each contains a gap: \cite[Proposition 5.3]{lehmann} does not hold in general (see \cite[\S 2.9]{eckl} for some discussion), while the proof of \cite[Proposition 3.4]{eckl} fails because the middle row of the commutative diagram is not necessarily exact.  This requires some additional corrections to the literature; see \cite[Corrigendum]{fujino}.
\end{remark}

\begin{remark}
Observe that Theorem~\ref{oguisocomputation} provides a counterexample to~\cite[Theorem 6.7, (7)]{lehmann}; it would be interesting to know whether for any pseudoeffective \(\R\)-divisor \(D\), there exist constants \(C_1\) and \(C_2\) for which
\[
C_1 m^{\kappa_\sigma^\R(D)} < h^0(X,\floor{m D} +A) < C_2 m^{\kappa_\sigma^\R(D)}.
\]
\end{remark}

\begin{remark}
Although for simplicity we have preferred explicit computations on the variety \(X\), the same strategy should suffice to compute the numerical dimension in many other contexts.  According to the Kawamata--Morrison cone conjecture, if \(X\) is a Calabi--Yau threefold, then for any big divisor class \(D\) there exists a pseudoautomorphism \(\phi : X \rat X\) such that \(\phi^\ast D\) lies in some fixed polyhedral subcone of \(\operatorname{Big}(X)\), where the volume can likely be computed explicitly.
\end{remark}

We now give a general computation in this vein, for another notion of numerical dimension, \(\nu_{\text{Vol}}\).  This invariant is similar to \(\kappa_\sigma\), but has two simplifying advantages: (i) one need not worry about the difference between \(\chi(D)\) and \(h^0(X,D)\) when \(X\) is not a Calabi--Yau, and (ii) it is not necessary to take the round-down of an \(\R\)-divisor, which in the case \(\rho(X) > 2\) could push the divisor out of the \(2\)-dimensional eigenspace for \(\phi^\ast\) spanned by \(\pD_+\) and \(\pD_-\).

\begin{definition}[\cite{lehmann}]
Suppose that \(X\) is a projective variety and \(D\) is a pseudoeffective divisor class on \(X\).  Fix an ample divisor \(A\).  The numerical dimension \(\nu_{\text{vol}}(D)\) is the largest integer \(k\) for which there exists a constant \(C\) satisfying
\[
Ct^{\dim X-k} < \vol(L + tA)
\]
for all \(t > 0\).  We also define \(\nu_{\text{vol}}^\R(D)\) to be th largest real number \(k\) with this property.
\end{definition}

\begin{lemma}
\label{vvolcompute}
Suppose that \(\phi : X \rat X\) is a pseudoautomorphism satisfying \(\lambda_1(\phi) > 1\).  Let \(\lambda_1 = \lambda_1(\phi)\) and \(\mu_1 = \lambda_1(\phi^{-1})\); it follows from the log-concavity of dynamical degrees that \(\mu_1 > 1\) as well.  Suppose that there exist 
a \(\lambda_1\)-eigenvector \(\pD_+\) for \(\phi^\ast\) and a \(\lambda_1\)-eigenvector \(\pD_-\) for \(\phi^{-1\ast}\) with the property that \(A = \pD_+ + \pD_-\) is ample.  Then

\[
\nu_{\text{vol}}^\R(\pD_+) = (\dim X) \left( 1 + \frac{\log \mu_1}{\log \lambda_1} \right)^{-1}
\]
\end{lemma}

\begin{proof}
Since \(\phi^\ast\) preserves the volume of a divisor,
\begin{align*}
  \vol(A) &= 
\vol(\pD_+ + \pD_-) =  \vol \left( (\phi^\ast)^n( \pD_+ + \pD_- ) \right) = \vol \left( \lambda_1^n \pD_+ + \mu_1^{-n} \pD_- \right) \\
  &= \vol \left( (\lambda_1^n - \mu_1^{-n}) \left( \pD_+ + \frac{\mu_1^{-n}}{\lambda_1^n - \mu_1^{-n}} (\pD_+ + \pD_-) \right) \right) \\
  &= (\lambda_1^n - \mu_1^{-n})^{\dim X} \vol \left( \pD_+ + \frac{\mu_1^{-n}}{\lambda_1^n - \mu_1^{-n}} (\pD_+ + \pD_-) \right).
\end{align*}
Taking \(A = \pD_+ + \pD_-\) and \(t_n = \frac{\mu_1^{-n}}{\lambda_1^n - \mu_1^{-n}} \), we find that
\[
\vol \left( \pD_+ + t_n A \right) = (\lambda_1^n - \mu_1^{-n})^{-\dim X} \vol(A) \\
  = C_n t_n^\nu,
\]
where
\begin{align*}
  \nu &= (\dim X) \left( 1 + \frac{\log \mu_1}{\log \lambda_1 } \right)^{-1} \\
  C_n &= \frac{(\lambda_1^n - \mu_1^{-n})^{-\dim X} \vol(A)}{\left( \frac{\mu_1^{-n}}{\lambda_1^n - \mu_1^{-n}} \right)^{ \left( (\dim X) \left( 1 + \frac{\log \mu_1}{\log \lambda_1 } \right)^{-1}\right) }}
\end{align*}

One may check that \(C_n\) is a decreasing function as \(n\) increases and that \(\lim_{n \to \infty} C_n = \vol(A)\).  In particular, for sufficiently large \(n\) we have
\[
\vol(A) t_n^\nu < \vol(\pD_+ + t_n A) < 2 \vol(A) t_n^\nu.
\]
Since \(\vol(\pD_+ + t A)\) is an increasing function in \(t\), this implies that there exists a constant \(C\) such that \(\vol(\pD_+ + t A) < C t^{\nu}\) for all \(0 < t < 1\), and so
\[
\nu_{\text{vol}}^\R(\pD_+) = \nu = (\dim X) \left( 1 + \frac{\log \mu_1}{\log \lambda_1 } \right)^{-1} = (\dim X) \left( 1 + \frac{\log \lambda_1(\phi^{-1})}{\log \lambda_1(\phi) } \right)^{-1}. \qedhere
\]
\end{proof}

In the example of this section, \(\lambda_1(\phi) = \lambda_1(\phi^{-1}) = \lambda\) and the formula yields \(\nu_{\text{vol}}(\pD_+) = \nicefrac{3}{2}\), which coincides with \(\kappa_\sigma^\R(\pD_+)\).

\begin{remark}
It is not at all clear that the quantity \((\dim X) \left( 1 + \frac{\log \lambda_1(\phi^{-1})}{\log \lambda_1(\phi) } \right)^{-1}\) should always be rational when \(\pD_+ + \pD_-\) is ample, although I am not aware of any relevant counterexamples.
\end{remark}

\begin{remark}
N.\ McCleerey has showed that \(\nu_{\text{vol}}(D) = \nubdpp(D)\) in several cases, e.g.\ when \(\nubdpp(D) = 0\) or \(\nubdpp(D) = \dim X - 1\) (when \(\dim X = 3\), this covers all cases except that of \(\nubdpp(D)\) which occurs for our main example)~\cite{mccleerey}.  It would also be interesting to know whether \(\kappa_\sigma(D) = \nu_{\text{vol}}(D)\) in general.
\end{remark}

When \(\phi\) is an automorphism with \(\lambda_1(\phi) > 1\) (rather than just a pseudoautomorphism), it is possible to give a more precise computation of the numerical dimension of the eigenvector in terms of the dynamical degrees of \(\phi\).  In this case, \(\pD_+\) is nef, and the different definitions of numerical dimension coincide; in particular, the value is always an integer.

Let \(J_k(\phi)\) denote the size of the largest Jordan block for \(\phi^\ast : N^k(X) \to N^k(X)\), and take \(\tilde J_k(\phi) = J_k(\phi) - 1\), so that for a general ample divisor \(H\) we have \((\phi^\ast)^m(H^k) \cdot H^{N-k} \sim \lambda_k(\phi)^m m^{\tilde J_k(\phi)}\).  Here by $\sim$ we mean that the left quantity is bounded above and below by multiples of the right one.

\begin{theorem}
  Suppose that \(\phi : X \to X\) is an automorphism with \(\lambda_1(\phi) > 1\) and that \(\pD_+\) is a leading eigenvector for \(\phi\), equal to \(\pD_+ = \lim_{n \to \infty} \frac{1}{\lambda_1(\phi)^n n^{\tilde J_1(\phi)}} (\phi^\ast)^n H\) for some ample \(H\).  Then \(
\nubdpp(\pD_+) = \kappa_\sigma(\pD_+) = \kappa_\nu(\pD_+)\) with
\[
\kappa_\sigma(\pD_+) = \max \set{a : \text{$\lambda_a(\phi) = \lambda_1(\phi)^a$ and $\tilde{J}_a(\phi) = a \tilde J_1(\phi)$}}.
\]
\end{theorem}

\begin{proof}
Let \(H\) be an ample class and \(N = \dim X\).  Then \(\pD_+^a \neq 0\) if and only if \(\pD_+^a \cdot H^{N-a} > 0\), and we compute
\begin{align*}
  \pD_+^a \cdot H^{N-a} &= \lim_{n \to \infty} \left( \frac{1}{\lambda_1(\phi)^{na} n^{a \tilde{J}_1(\phi)}}(\phi^\ast)^n (H)^a \right) \cdot H^{N-a} \\
  &= \lim_{n \to \infty} \frac{1}{\lambda_1(\phi)^{na}  n^{a \tilde{J}_1(\phi)}}  \left( (\phi^\ast)^n (H)^a \cdot H^{N-a} \right) \\
                        &= \lim_{n \to \infty} \frac{1}{\lambda_1(\phi)^{na}  n^{a \tilde{J}_1(\phi)}} \left( \lambda_a(\phi)^n n^{\tilde{J}_a(\phi)} \right) \\
  &= \lim_{n \to \infty} \left( \frac{\lambda_a(\phi)}{\lambda_1(\phi)^a} \right) \left(n^{\tilde J_a(\phi) - a \tilde J_1(\phi)} \right).
\end{align*}
Consequently \(\pD_+^a \cdot H^{N-a} > 0\) if \(\lambda_a(\phi) = \lambda_1(\phi)^a\) and \(\tilde J_a(\phi) = a \tilde J_1(\phi)\).  The claim follows. (Note that the first equality always holds if $=$ is changed to $\leq$, by log concavity of dynamical degrees; the same is true of the second in the case that the first is an equality.)
\end{proof}

\begin{example}
Suppose that \(X\) is a hyper-K\"ahler manifold of dimension \(N = 2m\) and that \(\phi : X \to X\) is an automorphism.  It is shown by Oguiso~\cite{oguisohk} that \(\lambda_a(\phi) = \lambda_1(\phi)^a\) for \(a \leq m\), so that \(\nu(\pD_+) = m\) in this case.
\end{example}

\section{Acknowledgments}

I am grateful to Brian Lehmann, Valentino Tosatti, and Osamu Fujino for discussions of these issues.  Mihai P\u{a}un and Izzet Coskun also provided useful feedback.  This work was supported by NSF Grant DMS-1700898, and the SageMath system was invaluable in carrying out a variety of computations.

\singlespacing
\bibliographystyle{amsplain}
\bibliography{refs}

\providecommand{\bysame}{\leavevmode\hbox to3em{\hrulefill}\thinspace}
\providecommand{\MR}{\relax\ifhmode\unskip\space\fi MR }
\providecommand{\MRhref}[2]{%
  \href{http://www.ams.org/mathscinet-getitem?mr=#1}{#2}
}
\providecommand{\href}[2]{#2}
\begin{thebibliography}{10}

\bibitem{bdpp}
S\'ebastien Boucksom, Jean-Pierre Demailly, Mihai P\u{a}un, and Thomas
  Peternell, \emph{The pseudo-effective cone of a compact {K}\"ahler manifold
  and varieties of negative {K}odaira dimension}, J. Algebraic Geom.
  \textbf{22} (2013), no.~2, 201--248.

\bibitem{bfj}
S\'{e}bastien Boucksom, Charles Favre, and Mattias Jonsson,
  \emph{Differentiability of volumes of divisors and a problem of {T}eissier},
  J. Algebraic Geom. \textbf{18} (2009), no.~2, 279--308.

\bibitem{cascinietal}
Paolo Cascini, Christopher Hacon, Mircea Musta\c{t}\u{a}, and Karl Schwede,
  \emph{On the numerical dimension of pseudo-effective divisors in positive
  characteristic}, Amer. J. Math. \textbf{136} (2014), no.~6, 1609--1628.

\bibitem{dinhsibony}
Tien-Cuong Dinh and Nessim Sibony, \emph{Une borne sup\'{e}rieure pour
  l'entropie topologique d'une application rationnelle}, Ann. of Math. (2)
  \textbf{161} (2005), no.~3, 1637--1644. \MR{2180409}

\bibitem{eckl}
Thomas Eckl, \emph{Numerical analogues of the {K}odaira dimension and the
  abundance conjecture}, Manuscripta Math. \textbf{150} (2016), no.~3-4,
  337--356.

\bibitem{fujino}
Osamu Fujino, \emph{On subadditivity of the logarithmic {K}odaira dimension},
  J. Math. Soc. Japan \textbf{69} (2017), no.~4, 1565--1581.

\bibitem{kawamataabundance}
Yujiro Kawamata, \emph{Abundance theorem for minimal threefolds}, Invent. Math.
  \textbf{108} (1992), no.~2, 229--246.

\bibitem{kawamatanumdimzero}
\bysame, \emph{On the abundance theorem in the case of numerical {K}odaira
  dimension zero}, Amer. J. Math. \textbf{135} (2013), no.~1, 115--124.

\bibitem{lazarsfeld}
Robert Lazarsfeld, \emph{Positivity in algebraic geometry. {I}}, vol.~48,
  Springer-Verlag, Berlin, 2004, Classical setting: line bundles and linear
  series.

\bibitem{lehmann}
Brian Lehmann, \emph{Comparing numerical dimensions}, Algebra Number Theory
  \textbf{7} (2013), no.~5, 1065--1100.

\bibitem{mccleerey}
Nicholas McCleerey, \emph{Volume of perturbations of pseudoeffective classes},
  preprint (2019).

\bibitem{nakayama}
Noboru Nakayama, \emph{Zariski-decomposition and abundance}, MSJ Memoirs,
  vol.~14, Mathematical Society of Japan, Tokyo, 2004.

\bibitem{oguisohk}
Keiji Oguiso, \emph{A remark on dynamical degrees of automorphisms of
  hyperk\"{a}hler manifolds}, Manuscripta Math. \textbf{130} (2009), no.~1,
  101--111.

\bibitem{oguiso}
\bysame, \emph{Automorphism groups of {C}alabi-{Y}au manifolds of {P}icard
  number 2}, J. Algebraic Geom. \textbf{23} (2014), no.~4, 775--795.

\end{thebibliography}

\end{document}